\RequirePackage[hyphens]{url}
\documentclass[a4paper,12pt]{article}

\usepackage{authblk}

\usepackage[utf8]{inputenc} 
\usepackage[T1]{fontenc}    
\usepackage{indentfirst}
\usepackage{lmodern}

\usepackage{mathtools}
\usepackage{verbatim}
\usepackage{booktabs}       
\usepackage{amsfonts}       
\usepackage{amsmath}
\usepackage{amssymb}
\usepackage{amsthm}
\usepackage{microtype}      
\usepackage[a4paper]{geometry} 
\usepackage{enumitem}
\usepackage{xcolor}

\PassOptionsToPackage{hyphens}{url}\usepackage[linktocpage=true]{hyperref}

\hypersetup{
           breaklinks=true,   
           colorlinks=true,   
           linkcolor=blue,
           pdfusetitle=true,  
        }

\DeclareMathOperator{\Tr}{Tr}
\DeclareMathOperator{\Gram}{Gram}
\DeclareMathOperator{\Span}{span}
\DeclareMathOperator{\codim}{codim}
\usepackage{graphicx}
\title{Path-connectedness and topological closure of some sets related to the non-compact Stiefel manifold}

\author{Nizar El Idrissi, Samir Kabbaj, and Brahim Moalige}

\newcommand{\Addresses}{{
  \bigskip
  \footnotesize

  \textbf{Nizar El Idrissi.}
  \par\nopagebreak Laboratoire : Equations aux dérivées partielles, Algèbre et Géométrie spectrales.
  \par\nopagebreak
  Département de mathématiques, faculté des sciences, université Ibn Tofail, 14000 Kénitra.\par\nopagebreak 
  \textit{E-mail address} : \texttt{nizar.elidrissi@uit.ac.ma}

  \medskip

  \textbf{Pr. Samir Kabbaj.} \par\nopagebreak Laboratoire : Equations aux dérivées partielles, Algèbre et Géométrie spectrales.
  \par\nopagebreak
  Département de mathématiques, faculté des sciences, université Ibn Tofail, 14000 Kénitra.\par\nopagebreak 
  \textit{E-mail address} : \texttt{samir.kabbaj@uit.ac.ma}
  
  \medskip
  
  \textbf{Pr. Brahim Moalige.}
  \par\nopagebreak Laboratoire : Equations aux dérivées partielles, Algèbre et Géométrie spectrales.
  \par\nopagebreak
  Département de mathématiques, faculté des sciences, université Ibn Tofail, 14000 Kénitra.\par\nopagebreak 
  \textit{E-mail address} : \texttt{brahim.moalige@uit.ac.ma}
}}

\theoremstyle{plain}

\newtheorem{proposition}{Proposition}[section]
\newtheorem{corollary}{Corollary}[section]
\newtheorem{lemma}{Lemma}[section]

\theoremstyle{definition}
\newtheorem{definition}{Definition}[section]
\newtheorem{example}{Example}[section]

\theoremstyle{remark}
\newtheorem{remark}{Remark}[section]

\makeatletter
\def\keywords{\xdef\@thefnmark{}\@footnotetext}
\makeatother

\makeatletter
\renewcommand*\env@matrix[1][*\c@MaxMatrixCols c]{%
  \hskip -\arraycolsep
  \let\@ifnextchar\new@ifnextchar
  \array{#1}}
\makeatother

\begin{document}
\newpage
\maketitle
\begin{abstract}
If $H$ is a Hilbert space, the non-compact Stiefel manifold $St(n,H)$ consists of independent $n$-tuples in $H$. In this article, we contribute to the topological study of non-compact Stiefel manifolds, mainly by proving two results on the path-connectedness and topological closure of some sets related to the non-compact Stiefel manifold. In the first part, after introducing and proving an essential lemma, we prove that $\bigcap_{j \in J} \left( U(j) + St(n,H) \right)$ is path-connected by polygonal paths under a condition on the codimension of the span of the components of the translating $J$-family. Then, in the second part, we show that the topological closure of $St(n,H) \cap S$ contains all polynomial paths contained in $S$ and passing through a point in $St(n,H)$. As a consequence, we prove that $St(n,H)$ is relatively dense in a certain class of subsets which we illustrate with many examples from frame theory coming from the study of the solutions of some linear and quadratic equations which are finite-dimensional continuous frames. Since $St(n,L^2(X,\mu;\mathbb{F}))$ is isometric to $\mathcal{F}_{(X,\Sigma,\mu),n}^\mathbb{F}$, this article is also a contribution to the theory of finite-dimensional continuous Hilbert space frames.
\end{abstract}




\tableofcontents

\section{Introduction}

Duffin and Shaeffer introduced in 1952 \cite{DuffinSchaeffer} the notion of a Hilbert space frame to study some deep problems in nonharmonic Fourier series. However, the general idea of signal decomposition in terms of elementary signals was known to Gabor \cite{Gabor} in 1946. The landmark paper of Daubechies, Grossmann, and Meyer \cite{DaubechiesGrossmannMeyer} (1986) accelerated the development of the theory of frames which then became more widely known to the mathematical community. Nowadays, frames have a wide range of applications in both engineering science and mathematics: they have found applications in signal processing, image processing, data compression, and sampling theory. They are also used in Banach space theory. Intuitively, a frame in a Hilbert space $K$ is an overcomplete basis allowing non-unique linear expansions, though technically, it must satisfy a double inequality called the frame inequality. There are many generalizations of frames in the literature, for instance frames in Banach space \cite{CasazzaHanLarson} or Hilbert C*-modules \cite{FrankLarson}. A general introduction to frame theory can be found in \cite{Casazza, Christensen}.

The space $\mathcal{F}_{(X,\Sigma,\mu),n}^\mathbb{F}$ of continuous frames indexed by $(X,\Sigma,\mu)$ and with values in $\mathbb{F}^n$ is isometric to the Stiefel manifold $St(n,L^2(X,\mu;\mathbb{F}))$. If $H$ is a Hilbert space, the non-compact Stiefel manifold $St(n,H)$ is the set of independent $n$-frames in $H$, where an independent $n$-frame simply denotes an independent $n$-tuple. Stiefel manifolds are studied in differential topology and are one of the fundamental examples in this area. Even though the theory of finite dimensional Stiefel manifolds is generally well-known \cite{James, Hatcher, Husemoller}, there are still some aspects under study \cite{Henkel, JurdjevicMarkinaLeite}. The theory of infinite dimensional Stiefel manifolds is less studied and some recent results can be found in \cite{BardelliMennucci, HarmsMennucci}.

There have been also many studies directly devoted to the geometry of frames and their subsets. Connectivity properties of some important subsets of the frame space $\mathcal{F}_{k,n}^\mathbb{F}$ were studied in \cite{CahillMixonStrawn, NeedhamShonkwiler}. Differential and algebro-geometric properties of these subsets were studied in \cite{DykemaStrawn, Strawn1, Strawn2, Strawn3} and chapter 4 of \cite{CasazzaKutyniok} respectively. A fiber bundle structure with respect to the $L^1$ and $L^\infty$ norms was established for continuous frames in \cite{Agrawal, AgrawalKnisley}. A notion of density for general frames analogous to Beurling density was introduced and studied in \cite{BalanCasazzaHeilLandau1}. Finally, connectivity and density properties were studied for Gabor \cite{BalanCasazzaHeilLandau2, ChristensenDengHeil, Heil, LabateWilson} and wavelet \cite{Bownik, GarrigosHernandezSikicSoriaWeissWilson, HanLarson, Speegle} frames. \\ \\
\textbf{Plan of the article.} This article is organized as follows. In section \ref{SectionPreliminaries}, we set some notations, introduce the definition of continuous Bessel and frame families and their basic properties in $\mathbb{F}^n$, and present Stiefel manifolds with an emphasis on their topological aspects. In section \ref{SectionPathConnectedness}, after introducing and proving an essential lemma, we prove that $\bigcap_{j \in J} \left( U(j) + St(n,H) \right)$ is path-connected by polygonal paths under a condition on the codimension of the span of the components of the translating $J$-family. Then, in section \ref{SectionRelativeDensity}, we show that the topological closure of $St(n,H) \cap S$ contains all polynomial paths contained in $S$ and passing through a point in $St(n,H)$. As a consequence, we prove that $St(n,H)$ is relatively dense in a certain class of subsets which we illustrate, in section \ref{SectionExamples}, with many examples from frame theory coming from the study of the solutions of some linear and quadratic equations which are finite-dimensional continuous frames (section \ref{SectionFrameSolutionLinQuadraticEq}).

\section{Preliminaries}
\label{SectionPreliminaries}

\subsection{Notations}
\label{SubsectionNotation}

The following notations are used throughout this article. \\
$\mathbb{N}$ denotes the set of natural numbers including 0 and $\mathbb{N}^* = \mathbb{N} \setminus \{0\}$. \\
We denote by $n$ an element of $\mathbb{N}^*$ and by $\mathbb{F}$ one of the fields $\mathbb{R}$ or $\mathbb{C}$. \\
If $K$ is a Hilbert space, we denote by $L(K)$ and $B(K)$ respectively the set of linear and bounded operators in $K$. $Id_K$ is the identity operator of $K$. \\
If $K$ is a Hilbert space, $m \in \mathbb{N}^*$, and $\theta_1, \cdots, \theta_m \in H$, the Gram matrix of $(\theta_1, \cdots, \theta_m)$ is the matrix $\Gram(\theta_1, \cdots, \theta_m)$ whose $k,l$-coefficient is $\Gram(\theta_1, \cdots, \theta_m)_{k,l} = \langle \theta_k, \theta_l \rangle$. \\
If $\sigma, \tau \in \mathbb{N}^*$, we denote by $M_{\sigma, \tau}(\mathbb{F})$ the algebra of matrices of size $\sigma \times \tau$ over the field $\mathbb{F}$. When $\sigma = \tau$, we denote this algebra $M_\sigma(\mathbb{F})$. \\
An element $x \in \mathbb{F}^n$ is a n-tuple $(x^1, \cdots, x^n)$ with $x^k \in \mathbb{F}$ for all $k \in [\![1,n]\!]$. \\
If $S \in L(\mathbb{F}^n)$, we denote by $[S] \in M_n(\mathbb{F})$ the matrix of $S$ in the standard basis of $\mathbb{F}^n$, and we write $I_n$ as a shorthand for $[Id_{\mathbb{F}^n}]$. \\
If $U = (u_x)_{x \in X}$ is a family in $\mathbb{F}^n$ indexed by $X$, then for each $k \in [\![1,n]\!]$, we denote by $U^k$ the family $(u_x^k)_{x \in X}$.

\subsection{Continuous frames in \texorpdfstring{$\mathbb{F}^n$}{Fn}}
\label{SubSectionFramesInFiniteDimension}

Let $K$ be a Hilbert space over $\mathbb{F}$ and $(X,\Sigma,\mu)$ a measure space.

\begin{definition} \cite{Christensen}
We say that a family $\Phi = (\varphi_x)_{x \in X}$ with $\varphi_x \in K$ for all $x \in X$ is a continuous frame in $K$ if \[ \exists \; 0 < A \leq B : \forall v \in K : A \| v \| ^2 \leq \int_X | \langle v , \varphi_x \rangle |^2 d\mu(x) \leq B \| v \|^2 \]
\end{definition}
A frame is tight if we can choose $A=B$ as frame bounds. A tight frame with bound $A=B=1$ is called a Parseval frame. A Bessel family is a family satisfying only the upper inequality. A frame is discrete if $\Sigma$ is the discrete $\sigma$-algebra and $\mu$ is the counting measure. If $(n,k) \in (\mathbb{N}^*)^2$, we denote by $\mathcal{F}_{(X,\Sigma,\mu),K}$, $\mathcal{F}_{(X,\Sigma,\mu),n}^\mathbb{F}$ and $\mathcal{F}_{k,n}^\mathbb{F}$ respectively the set of $(X,\Sigma,\mu)$-continuous frames with values in $K$, the set of $(X,\Sigma,\mu)$-continuous frames with values in $\mathbb{F}^n$, and the set of discrete frames indexed by $[\![1,k]\!]$ and with values in $\mathbb{F}^n$.

If  $U = (u_x)_{x \in X}$ with $u_x \in K$ for all $x \in K$ is a continuous Bessel family in $K$, we define its analysis operator $T_U : K \to L^2(X,\mu;\mathbb{F})$ by
\[ \forall v \in K : T_U(v) := (\langle v , u_x \rangle)_{x \in X}. \]
The adjoint of $T_U$ is an operator $T_U^* : L^2(X,\mu;\mathbb{F}) \to K$ given by
\[ \forall c \in L^2(X,\mu;\mathbb{F}) : T_U^*(c) = \int_X c(x) u_x d\mu(x). \]
The composition $S_U = T_U^* T_U : K \to K$ is given by
\[ \forall v \in K : S_U(v) = \int_X \langle v , u_x \rangle u_x d\mu(x) \]
and called the frame operator of $U$. Since $U$ is a Bessel family, $T_U$, $T_U^*$, and $S_U$ are all well defined and continuous. If $U$ is a frame in $K$, then $S_U$ is a positive self-adjoint operator satisfying $0 < A \leq S_U \leq B$ and thus, it is invertible. 

We now recall a proposition preventing that a frame belongs to $L^2(X,\mu;K)$ when $\dim(K)=\infty$. Here the set $L^2(X,\mu;K)$ refers to Bochner square integrable (classes) of functions in $\mathcal{M}(X;K)$, where the latter refers to the set of measurable functions from $X$ to $K$. It explains why we only study the $L^2$ topology of frame subspaces in the finite dimensional case.

\begin{proposition}
Let $K$ be a Hilbert space with $\dim{K} = \infty$. Then $\mathcal{F}_{(X,\Sigma,\mu),K} \cap L^2(X,\mu;K) = \emptyset$
\end{proposition}

\begin{proof}
Let $\Phi = (\varphi_x)_{x \in X} \in \mathcal{F}_{(X,\Sigma,\mu),K} \cap L^2(X,\mu;K)$. Let $\{e_m\}_{m \in M}$ be an orthonormal basis of $K$. We have
\begin{align*} 
\Tr(S_\Phi) = \Tr(T_\Phi^*T_\Phi) &= \sum_{m \in M} \|T(e_m) \|^2 \\
&= \sum_{m \in M} \int_X | \langle e_m , \varphi_x \rangle |^2 d\mu(x) \\
&= \int_X \left( \sum_{m \in M} | \langle e_m , \varphi_x \rangle |^2 \right) d\mu(x) \\
&= \int_X \| \varphi_x \|^2 d\mu(x).
\end{align*}
Since $\Phi \in \mathcal{F}_{(X,\Sigma,\mu),K}$, there exists a constant $A>0$ such that 
\[ S_\Phi \geq A \cdot Id, \]
so 
\[ \int_X \| \varphi_x \|^2 d\mu(x) = Tr(S_\Phi) = +\infty \]
since $\dim(K)=\infty$. Hence $\Phi \notin L^2(X,\mu;K)$.
\end{proof}
From now on, we consider $K = \mathbb{F}^n$. In what follows, we will recall some elementary facts about Bessel sequences and frames in this setting.

\begin{proposition} 
A family $U = (u_x)_{x \in X}$ with $u_x \in \mathbb{F}^n$ for all $x \in X$ is a continuous Bessel family if and only if it belongs to $L^2(X,\mu,\mathbb{F}^n)$.
\label{PropBesselL2Ck}
\end{proposition}

\begin{proof}
($\Rightarrow$) Suppose that $U = (u_x)_{x \in X}$ is a continuous Bessel family. For each $k \in [\![1,n]\!]$, denote by $e_k$ the $k$-th vector of the standard basis of $\mathbb{F}^n$. \\
Applying the definition to the vector $e_k$, we have for each $k \in \{1, \cdots, n\}$ : $\| U^k \|_{L^2(X,\mu,\mathbb{F})}^2 < \infty$, and so
\[ \| U \|_{L^2(X,\mu,\mathbb{F}^n)}^2 = \sum_{k=1}^n \| U^k \|_{L^2(X,\mu,\mathbb{F})}^2 < \infty, \]
which implies $U \in L^2(X,\mu;\mathbb{F}^n)$. \\
($\Leftarrow$) Suppose that $U = (u_x)_{x \in I} \in L^2(X,\mu;\mathbb{F}^n)$. We have
\[ \forall v \in \mathbb{F}^n : \int_{x \in X} |\langle v , u_x \rangle |^2 d\mu(x) \leq \| U \|_{L^2(X,\mu;\mathbb{F}^n)}^2 \lVert v \rVert^2 < \infty \]
by the Cauchy-Schwarz inequality, which implies that $U = (u_x)_{x \in X}$ is a continuous Bessel family.
\end{proof}

\begin{lemma}
If $U \in L^2(X,\mu;\mathbb{F}^n)$, then $[S_U] = \Gram(U^1, \cdots, U^n)$.
\end{lemma}

\begin{proof}
Let $(e_k)_{k \in [\![1,n]\!]}$ the standard basis of $\mathbb{F}^n$. Let $i,j \in [\![1,n]\!]$. Then \[ [S_U]_{i,j} = \langle S e_j , e_i \rangle = \int_X \langle e_j , u_x \rangle \langle u_x , e_i \rangle d\mu(x) =  \int_X \overline{u_x^j} u_x^i d\mu(x) = \langle U^i , U^j \rangle. \]
\end{proof}

\begin{proposition} \cite{Christensen}
\label{PropCharFrames}
Suppose $\Phi = (\varphi_x)_{x \in X}$ is a family in $\mathbb{F}^n$. Then
\begin{align*}
\Phi \text{ is a continuous frame} &\Leftrightarrow \Phi \in L^2(X,\mu;\mathbb{F}^n) \text{ and } S_\Phi \text{ is invertible} \\
&\Leftrightarrow \Phi \in L^2(X,\mu;\mathbb{F}^n) \mbox{ and } \det(\Gram(\Phi^1,\cdots,\Phi^n)) > 0  \\
&\Leftrightarrow \Phi \in L^2(X,\mu;\mathbb{F}^n) \mbox{ and } \{\Phi^1, \cdots, \Phi^n\} \mbox{ is free}.
\end{align*}
\end{proposition}

\begin{proposition} \cite{Christensen}
\label{PropCharTightFrames}
Suppose $\Phi = \{\varphi_x\}_{x \in X}$ is a family in $\mathbb{F}^n$ and let $a > 0$. Then
\begin{align*}
\Phi \text{ is a measurable a-tight frame} &\Leftrightarrow \Phi \in L^2(X,\mu;\mathbb{F}^n) \mbox{ and } S_\Phi = aI_n \\
&\Leftrightarrow \Phi \in L^2(X,\mu;\mathbb{F}^n) \mbox{ and } \Gram(\Phi^1,\cdots,\Phi^n) = aI_n \\
&\Leftrightarrow \Phi \in L^2(X,\mu;\mathbb{F}^n) \mbox{ and } (\Phi^1,\cdots,\Phi^n) \\
& \text{is an orthogonal family of } L^2(X,\mu;\mathbb{F}) \\
& \text{and } (\forall i \in [\![1,n]\!] : \lVert \Phi^i \rVert = \sqrt{a}).
\end{align*}
\end{proposition}

\begin{example}
Define $\varphi_m^1 = \frac{1}{m} e^{2 \pi i a m }$ and $\varphi_m^2 = \frac{1}{m} e^{2 \pi i b m }$ with $a,b$ two real numbers such that $a-b$ is not an integer. Then $\Phi^1 = (\varphi_m^1)_{m \in \mathbb{N}}$ and $\Phi^2 = (\varphi_m^2)_{m \in \mathbb{N}}$ are square summable with sum $\frac{\pi^2}{6}$. Since the sequences $\Phi^1$ and $\Phi^2$ are not proportional due to the constraint on $a$ and $b$,  it follows by \ref{PropCharFrames} that $\Phi$ is a discrete frame in $\mathbb{C}^2$. It is not however a tight frame since $\Phi^1$ and $\Phi^2$ are not orthogonal.
\end{example}

\subsection{Basic topological properties of \texorpdfstring{$St(n,H)$}{St(n,H)} and \texorpdfstring{$St_o(n,H)$}{Sto(n,H)}}
\label{SubSectionBasicTopPropertiesSt(n,H)}

In this subsection, we introduce $St(n,H)$ and $St_o(n,H)$ as well as some of their basic topological properties. We recall that $n$ is a fixed element of $\mathbb{N}^*$. If $H$ is a Hilbert space, then $St(n,H)$ is non-empty exactly when $\dim(H) \geq n$. In the following, we will always suppose this condition.

\begin{definition}
The \textit{non-compact Stiefel manifold} of independent $n$-frames in $H$ is defined by ${St(n,H) := \{h = (h_1,\cdots,h_n) \in H^n : \{h_1,\cdots,h_n\} \text{ is free}\}}$. The \textit{Stiefel manifold of orthonormal $n$-frames in $H$} is defined by ${St_o(n,H) := \{h = (h_1,\cdots,h_n) \in H^n : \{h_1,\cdots,h_n\} \text{ is an orthonormal system}\}}$.
\end{definition}

\begin{proposition}
\label{PropStiefelIsometricFrames}  
We have
\begin{enumerate}
\item $St(n,L^2(X,\mu;\mathbb{F}))$ is isometric to $\mathcal{F}_{(X,\Sigma,\mu),n}^\mathbb{F}$.
\item $St_o(n,L^2(X,\mu;\mathbb{F}))$ is isometric to the set of continuous $(X,\Sigma,\mu)$-Parseval frames with values in $\mathbb{F}^n$.
\end{enumerate}
\end{proposition}

\begin{proof}
Define 
\[ \textbf{\text{Transpose}} : \begin{cases} L^2(X,\mu;\mathbb{F}^n) &\to L^2(X,\mu;\mathbb{F})^n \\ F = (f_x)_{x \in X} &\mapsto ((f_x^1)_{x \in X},\cdots,(f_x^n)_{x \in X}) \end{cases}. \]
Then \textbf{Transpose} is clearly an isometry, and it sends $\mathcal{F}_{(X,\Sigma,\mu),n}^\mathbb{F}$ to $St(n,L^2(X,\mu;\mathbb{F}))$ and $St_o(n,L^2(X,\mu;\mathbb{F}))$ to the set of continuous $(X,\Sigma,\mu)$-Parseval frames with values in $\mathbb{F}^n$ by propositions \ref{PropCharFrames} and \ref{PropCharTightFrames} respectively. 
\end{proof}

\begin{remark}
\label{RemarkIsometry}
Because of proposition \ref{PropStiefelIsometricFrames}, the reader should keep in mind that the following topological properties and the new results of this article are also shared, for any measure space $(X,\Sigma,\mu)$, by $\mathcal{F}_{(X,\Sigma,\mu),n}^\mathbb{F}$ or the set of continuous $(X,\Sigma,\mu)$-Parseval frames with values in $\mathbb{F}^n$, depending on the context.
\end{remark}

\begin{proposition}
\label{PropFramesOpenClosed}
We have
\begin{enumerate}
\item $St(n,H)$ is open in $H^n$.
\item $St_o(n,H)$ is closed in $H^n$.
\end{enumerate}
\end{proposition}

\begin{proof} 
\begin{enumerate}
\item $St(n,H)$ is open because $St(n,H) = (\det \circ \Gram)^{-1}((0,\infty))$.
\item $St_o(n,H)$ is closed because $St_o(n,H) = \Gram^{-1}(I_n)$.
\end{enumerate}
\end{proof}

By joining continuously each element of $St(n,H)$ to its corresponding Gram-Schmidt orthonormalized system in $St_o(n,H)$, we can prove

\begin{proposition} $St_o(n,H)$ is a deformation retract of $St(n,H)$.
\end{proposition}

Concerning the connectedness properties of Stiefel manifolds, we have the following

\begin{definition}
Let $X$ be a topological space and $m \in \mathbb{N}$. Then $X$ is said to be $m$-connected if its homotopy groups $\pi_i(X,b)$ are trivial for all $i \in [\![0,m]\!]$ and $b \in X$ (of course, $\pi_0(X,b)$ is not a group but merely a pointed set).
\end{definition}

\begin{proposition} (see pp. 382-383 of \cite{Hatcher}) We have
\begin{enumerate}
\item $St(n,\mathbb{R}^k)$ is $(k-n-1)$-connected.
\item $St(n,\mathbb{C}^k)$ is $(2k-2n)$-connected.
\item If $H$ is infinite dimensional, then $St(n,H)$ is contractible.
\end{enumerate}
\end{proposition}

Moreover, we have

\begin{proposition} 
If $H$ is infinite dimensional, then $St(n,H)$ is contractible.
\end{proposition}

A proof can be found in \href{https://math.stackexchange.com/questions/199133/contractibility-of-the-sphere-and-stiefel-manifolds-of-a-separable-hilbert-space/200481}{this math.stackexchange thread}.

The following proposition asserts the density of $St(n,H)$ in $H^n$. A corollary of one of our results in this article (corollary \ref{CorSt(n,H)CapSDenseInS}) gives a generalization of this proposition.

\begin{proposition}
\label{PropSt(n,H)Dense}
$St(n,H)$ is dense in $H^n$.
\end{proposition}

\begin{proof}
Consider $h = (h_1,\cdots,h_n) \in H^n$. Pick some $\theta = (\theta_1,\cdots,\theta_n) \in St(n,H)$. Let $\gamma$ be the straight path connecting $\theta$ to $h$, i.e. for each $t \in [0,1] : \gamma(t) = th + (1-t)\theta \in H^n$. Let $\Gamma(t) = \det(\Gram ((\gamma(t)_1, \cdots,\gamma(t)_n)))$. Clearly, $\Gamma(t)$ is a polynomial function in $t$ which satisfies $\Gamma(0) \neq 0$ since $\theta \in St(n,H)$. Therefore 
\[ \Gamma(t) \neq 0 \text{ except for a finite number of } t's. \] 
Moreover, 
\[ ||\gamma(t) - h||_{H^n}^2 = \sum_{k=1}^n ||\gamma(t)_k - h_k||_{H}^2 = \sum_{i=1}^n |1-t|^2 ||\theta_k - h_k||_{H}^2 \to 0 \mbox{ when t } \to 1 \]
Hence, there exists $t \in [0,1]$ such that $\gamma(t)$ is close to $h$ and $\Gamma(t) \neq 0$, and so $\gamma(t) \in St(n,H)$.
\end{proof}

We also include the following proposition on the differential structure of the Stiefel manifolds.

\begin{proposition} \cite{James}
We have
\begin{enumerate}
\item $St(n,\mathbb{R}^k)$ is a real manifold of dimension $nk$.
\item $St_o(n,\mathbb{R}^k)$ is a real manifold of dimension $nk - \frac{n(n+1)}{2}$.
\item $St(n,\mathbb{C}^k)$ is a real manifold of dimension $2nk$.
\item $St_o(n,\mathbb{C}^k)$ is a real manifold of dimension $2nk-n^2$.
\item If $\dim(H)=\infty$, then $St(n,H)$ and $St_o(n,H)$ are Hilbert manifolds of infinite dimension.
\end{enumerate}
\end{proposition}

We ask the reader to keep in mind remark \ref{RemarkIsometry} when reading the remaining parts of this article.

\section{Path-connectedness of $\bigcap_{j \in J} \left( U(j) + St(n,H) \right)$}
\label{SectionPathConnectedness}

\begin{definition}
Let $E$ be a topological vector space and $\gamma : [0,1] \to E$ be a continuous path. We say that $\gamma$ is a \textit{polygonal path} if there exist $q \in \mathbb{N}^*$, $(e_k)_{k \in [\![1,q+1]\!]}$ a finite sequences with $e_k \in E$ for all $k \in [\![1,q+1]\!]$, and $(\gamma_k)_{k \in [\![1,q]\!]}$ a finite sequence of (continuous) straight paths with $\gamma_k = \begin{cases} [0,1] &\to E \\ t &\mapsto (1-t)e_k + te_{k+1} \end{cases}$ such that $\gamma = \gamma_1 * \cdots * \gamma_q$, where $*$ is the path composition operation. We say that a subset $S \subseteq E$ is polygonally connected if every two points of $S$ are connected by a polygonal path. 
\end{definition}

In the following, when we say that $S \subseteq E$ is polygonally connected, we mean that each two points of $S$ are connected by a polygonal path of the type $\gamma_1 * \gamma_2$ where $\gamma_1$ and $\gamma_2$ are two straight paths.

Before we prove the main proposition of this section, let's prove a useful lemma.

\begin{lemma}
\phantomsection
\label{LemmaSpanFrames}
\begin{enumerate}
\item Suppose we have a family $(a(j))_{j \in J}$ indexed by $J$ where each $a(j)$ belongs to $St(n,H)$. Then if $(a(j)_k)_{j \in J, k \in [\![1,n]\!]}$ is free, we have
\[ \Span( \{a(j)\}_{j \in J} ) \setminus \{0\} \subseteq St(n,H) \].
\item Suppose we have a family indexed by $J$ where each $a(j)$ belongs to $St_o(n,H)$ for all $j \in J$. Then if $(a(j)_k)_{j \in J, k \in [\![1,n]\!]}$ is an orthonormal system, we have
\[ \{ x \in \Span( \{a(j)\}_{j \in J} ) : \| x \| = \sqrt{n} \} \subseteq St_o(n,H) \]
\end{enumerate}
\end{lemma}

\begin{proof}
\begin{enumerate}
\item Let $h = (h_1,\cdots,h_n) \in \Span( \{a(j)\}_{j \in J} ) \setminus \{0\}$. We can write $h = \sum_{u=1}^r \lambda_u a(j_u)$ with $\lambda_u \in \mathbb{F}$ for all $u \in [\![1,r]\!]$ and the $\lambda_u$'s are not all zeros. We need to show that $(h_1,\cdots,h_n)$ is an independent system. Suppose otherwise $\sum_{k=1}^n c_k h_k = 0$. This means that $\sum_{k=1}^n c_k ( \sum_{u=1}^r \lambda_u a(j_u)_k) = 0$, and so $\sum_{k=1}^n \sum_{u=1}^r (\lambda_u c_k) a(j_u)_k = 0$. Since $\bigcup_{j \in J} \bigcup_{k \in [\![1,n]\!]} \{a(j)_k\}$ is free, we deduce that $\lambda_u c_k = 0$ for all $u \in [\![1,r]\!]$ and $k \in [\![1,n]\!]$, which implies that $c_k = 0$ for all $k \in [\![1,n]\!]$ since the $\lambda_u$'s are not all zeros. Therefore, $h \in St(n,H)$.

\item Let $h = (h_1,\cdots,h_n) \in \Span(\{a(j)\}_{j \in J})$ such that $\|h\| = \sqrt{n}$. We can write $h = \sum_{u=1}^r \lambda_u a(j_u)$ with $\lambda_u \in \mathbb{F}$ for all $u \in [\![1,r]\!]$. We need to show that $\langle h_k , h_l \rangle = \delta_{k,l}$ for all $k,l \in [\![1,n]\!]$. For $k \neq l$, we have : $\langle h_k , h_l \rangle = \langle \sum_{u=1}^r \lambda_u a(j_u)_k , \sum_{u=1}^r \lambda_u a(j_u)_l \rangle = \sum_{u=1}^r \sum_{v=1}^r \lambda_u \overline{\lambda_v} \langle a(j_u)_k , a(j_v)_l \rangle = 0$ since $\bigcup_{j \in J} \bigcup_{k \in [\![1,n]\!]} \{a(j)_k\}$ is an orthogonal system. Moreover, $\| h_k \|^2 = \sum_{u=1}^r \sum_{v=1}^r \lambda_u \overline{\lambda_v} \langle a(j_u)_k , a(j_v)_k \rangle = \sum_{u=1}^r |\lambda_u|^2 \| a(j_u)_k \| ^2 = \sum_{u=1}^r |\lambda_u|^2$ since $\bigcup_{j \in J} \bigcup_{k \in [\![1,n]\!]} \{a(j)_k\}$ is an orthonormal system. By hypothesis, $n = \| h \| ^ 2 = \sum_{k=1}^n \| h_k \|^2 = n (\sum_{u=1}^r |\lambda_u|^2)$, so $\| h_k \|^2 = \sum_{u=1}^r |\lambda_u|^2 = 1$ for all $k \in [\![1,n]\!]$ as desired.
\end{enumerate}
\end{proof}

\begin{proposition}
Let $H$ be a Hilbert space with $\dim(H) \geq n$, $J$ an index set and $(U(j))_{j \in J}$ a family with $U(j) \in H^n$ for all $j \in J$. If $\codim_H(\Span(\{u(j)_k : j \in J, k \in [\![1,n]\!]\})) \geq 3n$, then $\bigcap_{j \in J} (U(j) + St(n,H))$ is polygonally-connected. 
\end{proposition}

\begin{proof}
Let $X = (x_1,\cdots,x_n)$ and $Y = (y_1,\cdots,y_n)$ in $\bigcap_{j \in J} (U(j) + St(n,H))$. \\
Let $(z_1, \cdots, z_n)$ be an independent family in $H$ such that
\begin{align*}
&\Span(\{z_k : k \in [\![1,n]\!]\}) \bigcap \\
&\quad \Span\left(\{x_k : k \in [\![1,n]\!]\} \bigcup \{y_k : k \in [\![1,n]\!]\} \bigcup \{u(j)_k : j \in J, k \in [\![1,n]\!]\} \right) = \{0\}
\end{align*}
This is possible since $\codim_H(\Span(\{u(j)_k : j \in J, k \in [\![1,n]\!]\}) \geq 3n$. \\
This ensures that we have for all $j \in J$
\[ \Span( \{-u(j)_k + z_k : k \in [\![1,n]\!] \}) \cap \Span(\{-u(j)_k + x_k : k \in [\![1,n]\!]\}) = \{0\}, \]
\[ \Span( \{-u(j)_k + z_k : k \in [\![1,n]\!] \}) \cap \Span(\{-u(j)_k + y_k : k \in [\![1,n]\!]\}) = \{0\}, \]
and 
\[ (-u(j)_k + z_k)_{k \in [\![1,n]\!]} \text{ is independent.} \]
We define the straight paths 
\[
\begin{cases} \gamma_1 &: [0,1] \to H^n \\
\gamma_2 &: [0,1] \to H^n \end{cases}
\]
by $\gamma_1(t) = tZ + (1-t)X$ and $\gamma_2(t) = tY + (1-t)Z$ respectively.  \\
We have $\gamma_1(0)=X$, $\gamma_1(1)=\gamma_2(0)=Z$, and $\gamma_2(1)=Y$. \\
Since for all $j \in J$
\[ \Span( \{-u(j)_k + z_k : k \in [\![1,n]\!] \}) \cap \Span(\{-u(j)_k + x_k : k \in [\![1,n]\!]\}) = \{0\}, \]
we have for all $t \in [0,1]$ and $j \in J$ 
\[ -U(j) + tZ + (1-t)X = t(-U(j) + Z) + (1-t)(-U(j) + X) \in St(n,H) \]
by lemma \ref{LemmaSpanFrames}, and so $\gamma_1(t) \in \bigcap_{j \in J} (U(j) + St(n,H))$. \\
Similarly, $\gamma_2(t) \in \bigcap_{j \in J} (U(j) + St(n,H))$ for all $t \in [0,1]$. \\
Composing $\gamma_1$ with $\gamma_2$, we see that $\bigcap_{j \in J} (U(j) + St(n,H))$ is polygonally connected.
\end{proof}

\section{Topological closure of $St(n,H) \cap S$}
\label{SectionRelativeDensity}

Before moving on, we need a definition and a small lemma.

\begin{definition}
Let $V$ be a $\mathbb{F}$-vector space, $q \in \mathbb{N}$, and $v,v' \in V$. We say that $\gamma : [0,1] \to V$ is a \textit{polynomial path up to reparametrization joining $v$ and $v'$} if there exist $q \in \mathbb{N}$, a finite sequence of vectors $(v^k)_{k \in [\![0,q]\!]}$ with $v^k \in V$ for all $k \in [\![0,q]\!]$ and a homeomorphism $\phi : [0,1] \to [a,b] \subseteq \mathbb{R}$ such that $\forall t \in [a,b] : \gamma(\phi^{-1}(t)) = \sum_{k=0}^q t^k v^k$, $\gamma(0)=v$ and $\gamma(1)=v'$. If $V$ is equipped with a topology, then we say that $\gamma$ is a \textit{continuous polynomial path} when it is continuous as a map from $[0,1]$ to $V$.
\end{definition}

\begin{remark}
Every expression of the form $\sum_{k=0}^q P_k(t)v^k$ where $v^k \in V$ and $P_k \in \mathbb{F}[X]$ for all $k \in [\![0,q]\!]$ can be written in the form $\sum_{k=0}^{q'} t^k w^k$ where $q' \in \mathbb{N}$ and $w^k \in V$ for all $k \in [\![0,q']\!]$ (group by increasing powers of $t$). Therefore there is no difference whether we define polynomial paths using the first expression or the second.
\end{remark}

\begin{lemma}
\label{LemmaPolynomialPathsBoundedDegreeInNormedVectorSpaceContinuous}
Let $E$ be a normed vector space, and $v,v' \in E$. Then each polynomial path $\gamma : [0,1] \to E$ up to reparametrization joining $v$ and $v'$ is continuous.
\end{lemma}

\begin{proof}
Let $\gamma : [0,1] \to E$ be a polynomial path up to reparametrization joining $v$ and $v'$. Hence we can write $\forall t \in [0,1] : \gamma(t) = \sum_{k=0}^q \phi(t)^k v^k$. The continuity of $\gamma$ follows from
\[ \lVert \gamma(t) - \gamma(t') \rVert \leq \sum_{k=0}^q \lVert v^k \rVert |\phi(t)^k - \phi(t')^k) | \]
which goes to 0 when $t$ goes to $t'$ by continuity of $\phi^k$ for all $k \in [\![0,q]\!]$.
\end{proof}

The following is our first original proposition in this section.

\begin{proposition}
\label{PropClosureSt(n,H)CapS}
Let $S \subseteq H^n$ and 
\begin{align*}
E := \{ \gamma : [0,1] \to S \text{ such that } &\gamma \text{ is a polynomial path up to reparametrization and } \\
&\exists a_\gamma \in [0,1] : \gamma(a_\gamma) \in St(n,H) \cap S \}.
\end{align*}
Then $\bigcup_{\gamma \in E} Range(\gamma) \subseteq \overline{St(n,H) \cap S}$
\end{proposition}

\begin{proof}
Let $\gamma \in E$.  We have $\forall t \in [a,b] : \gamma(\phi^{-1}(t)) = \sum_{k=0}^q t^k V^k$. \\
Let $\Gamma(t) := \det(\Gram((\gamma(\phi^{-1}(t))_1, \cdots, \gamma(\phi^{-1}(t))_n)))$ for all $t \in [a,b]$. \\
Since for all $i,j \in [\![1,n]\!]$ and $t \in [a,b]$
\begin{align*}
\langle \gamma(\phi^{-1}(t))_i , \gamma(\phi^{-1}(t))_j \rangle &= \left \langle \sum_{k=0}^q t^k v_i^k , \sum_{k=0}^q t^k v_j^k \right\rangle \\
&= \sum_{k,k'=0}^q \left\langle v_i^k, v_j^{k'} \right\rangle t^{k+k'}
\end{align*} 
is a polynomial function in $t \in [a,b]$, and the determinant of a matrix in $M_{n,n}(\mathbb{F})$ is a polynomial function in its coefficients, $\Gamma(t)$ is a polynomial function in $t \in [a,b]$ which satisfies $\Gamma(\phi(a_\gamma)) \neq 0$ since $\gamma(a_\gamma) \in St(n,H)$. Therefore
\[ \Gamma(t) \neq 0 \text{ for t in a cofinite set } L \subseteq [a,b]. \]
Hence $Range(\gamma) \setminus \{\gamma(\phi^{-1}(t))\}_{t \in [a,b] \setminus L} \subseteq St(n,H) \cap S \subseteq \overline{ St(n,H) \cap S }$.
Since $[a,b] \setminus L$ is finite, the continuity of $\gamma$ (see lemma \ref{LemmaPolynomialPathsBoundedDegreeInNormedVectorSpaceContinuous}) at $\{ \phi^{-1}(t))\}_{t \in [a,b] \setminus L}$ implies 
\[ Range(\gamma) = \overline{Range(\gamma)} = \overline{Range(\gamma) \setminus \{\gamma(\phi^{-1}(t))\}_{t \in [a,b] \setminus L} } \subseteq \overline{ St(n,H) \cap S }. \]
This being true for all $\gamma \in E$, the result follows.
\end{proof}

The following is a corollary.

\begin{corollary}
\label{CorSt(n,H)CapSDenseInS}
Let $S \subseteq H^n$ such that for all $U \in S$ there exists a polynomial path up to reparametrization connecting $U$ to some $\Theta(U) \in St(n,H) \cap S$ and contained in $S$. Then $St(n,H) \cap S$ is dense in $S$.
\end{corollary}

\begin{proof}
For all $U \in S$, there exists by the hypothesis $\gamma \in E$ such that $\gamma(0) = U$. By proposition \ref{PropClosureSt(n,H)CapS}, we have $Range(\gamma) \subseteq \overline{St(n,H) \cap S}$. It follows that $U \in \overline{St(n,H) \cap S}$ since $U \in Range(\gamma)$.
\end{proof}

This result shows the abundance of independent $n$-frames not only in $H^n$ but also in many subsets $S$ of the form of corollary \ref{CorSt(n,H)CapSDenseInS}. Importantly, notice that for $S \neq \emptyset$, there should exist at least one $\Theta \in St(n,H) \cap S$ (i.e. $St(n,H) \cap S \neq \emptyset$) for the result to follow. 

As an example, this is true when $S$ is a star domain of $H^n$ with respect to some $\Theta \in St(n,H) \cap S$; if it is a convex subset of $H^n$ and contains some $\Theta \in St(n,H) \cap S$; and if it is in particular an affine subspace containing some $\Theta \in St(n,H) \cap S$.

In the next section, we will find sufficient conditions under which some sets of the form $f^{-1}(\{d\}) \subseteq L^2(X,\mu;\mathbb{F}^n)$ where $f$ is some linear or quadratic function contain a continuous frame. This will allow us to apply corollary \ref{CorSt(n,H)CapSDenseInS} to these examples, which will be done in section \ref{SectionExamples}.

\section{Existence of solutions of some linear and quadratic equations which are finite-dimensional continuous frames}
\label{SectionFrameSolutionLinQuadraticEq}

In this section, we show how to construct continuous finite-dimensional frames that are mapped to a given element by a linear operator or a quadratic function. In other words, we show the existence of frames in the inverse image of singletons by these functions.

\subsection{Linear equations}

\begin{proposition}
\label{PropT^(-1)(d)ContainsIndependentNTuplesTArbitraryLinearFormCondDimV}
Let $n \in \mathbb{N}^*$, $V$ an $\mathbb{F}$-vector field of dimension $\geq n$ and $T : V^n \to \mathbb{F}$ a non-zero linear form. Then for all $d \neq 0$, there exists $(a_1,\cdots,a_n) \in V^n$ such that the system $(a_1,\cdots,a_n)$ is free and $T(a_1,\cdots,a_n)=d$.
\end{proposition}

The proof relies on the following lemma, which may be of independent interest.

\begin{lemma}
\label{LemmaAnyNTupleIsLinearCombinationOfFreeeSystems}
Let $n \in \mathbb{N}^*$ and $V$ be a vector space of dimension $\geq n$. Let $(x_1,\cdots,x_n) \in V^n \setminus \{(0,\cdots,0)\}$. Then there exist $e \in \mathbb{N}^*$ with $e = \begin{cases} 1 \text{ if } k=n \\ 2 \text{ if } k \in [\![1,n-1]\!] \end{cases}$ where $k = \dim(\Span\{x_1, \cdots, x_n\})$, and $e$ independent systems $(a^u_1,\cdots,a^u_n)$ in $V$ for $u \in [1,e]$ such that $(x_1,\cdots,x_n) = \sum_{u=1}^e (a^u_1,\cdots,a^u_n)$.
\end{lemma}

\begin{proof}
Without loss of generality, suppose that $(x_1, \cdots, x_k)$ is free where $k = \dim(\Span\{x_1, \cdots, x_n\}) \geq 1$. Let $(b_i)_{i \in [\![1,n-k]\!]}$ be an $(n-k)$-tuple of vectors of $V$ such that the system $(x_1,\cdots,x_k,b_1,\cdots,b_{n-k})$ is free.
\begin{itemize}
\item If $k=n$, then $(x_1, \cdots, x_n)$ is already free and so we can choose $e=1$ and $a^1_i = x_i$ for all $i \in [\![1,n]\!]$.
\item If $k \in [\![1,n-1]\!]$, we can write 
\[ (x_1, \cdots, x_n) = (\frac{x_1}{2},\cdots,\frac{x_k}{2},b_1,\cdots,b_{n-k}) + (\frac{x_1}{2},\cdots,\frac{x_k}{2},x_{k+1}-b_1,\cdots,x_n - b_{n-k}). \]
 $(\frac{x_1}{2},\cdots,\frac{x_k}{2},b_1,\cdots,b_{n-k})$ is free by assumption, and we can easily show that $(\frac{x_1}{2},\cdots,\frac{x_k}{2},x_{k+1}-b_1,\cdots,x_n - b_{n-k})$ is free by expressing $x_{k+1},\cdots,x_n$ in terms of $x_1,\cdots,x_k$. Hence we can choose $e=2$ and $a^1_i = a^2_i = \frac{x_i}{2}$ for all $i \in [\![1,k]\!]$, $a^1_i = b_{i-k}$, and $a^2_i = x_i - b_{i-k}$ for all $i \in [\![k+1,n]\!]$. 
\end{itemize}
\end{proof}

\begin{proof} (of proposition \ref{PropT^(-1)(d)ContainsIndependentNTuplesTArbitraryLinearFormCondDimV}) \\
Let's show that there exists a free system $(a_1,\cdots,a_n)$ such that $T(a_1,\cdots,a_n) \neq 0$. Suppose to the contrary that $T(a_1,\cdots,a_n) = 0$ for all free systems $(a_1,\cdots,a_n)$ in $V$. Let $(x_1,\cdots,x_n) \in V^n \setminus \{(0,\cdots,0)\}$. By lemma \ref{LemmaAnyNTupleIsLinearCombinationOfFreeeSystems}, $(x_1,\cdots,x_n)$ decomposes as a finite sum of free systems. By linearity of $T$, we thus have $T(x_1,\cdots,x_n)=0$. Hence $T$ is the zero form, a contradiction. Hence there exists a free system $(a_1,\cdots, a_n)$ such that $T(a_1,\cdots,a_n) \neq 0$. Then $\frac{d}{T(a_1,\cdots,a_n)}(a_1,\cdots,a_n)$ satisfies the requirement.
\end{proof}

\begin{corollary}
\label{CorT^(-1)(d)ContainsFrameTArbitraryLinearFormCondDimL2X}
Let $n \in \mathbb{N}^*$ and $T : L^2(X,\mu;\mathbb{F}^n) \to \mathbb{F}$ be a non-zero linear form. \\
Suppose that $\dim(L^2(X,\mu;\mathbb{F})) \geq n$. \\
Then for all $d \neq 0$, there exists a continuous frame $\Phi = (\varphi_x)_{x \in X}$ such that $T(\Phi)=d$.
\end{corollary}

\begin{proposition}
\phantomsection
\label{PropT^(-1)(d)ContainsFrameEachCordTDependsSeparatelyOnOneVarCondDimL2X}
Let $n \in \mathbb{N}^*$, $V$ be a vector space over $\mathbb{F}$, and $S : V \to \mathbb{F}$ be a non-zero linear form. \\
Define the linear operator 
\[ T : \begin{cases} V^n &\to \mathbb{F}^n \\ (x_1,\cdots,x_n) &\mapsto (S(x_1), \cdots, S(x_n)) \end{cases}. \]
For all $d \in \mathbb{F}^n$, suppose that
\begin{itemize}
\item if $d \neq 0$, then $\dim(V) \geq n$,
\item if $d = 0$ then $\dim(V) \geq n+1$.
\end{itemize}
Then $T^{-1}(\{d\})$ contains an independent n-tuple $(a_1,\cdots,a_n)$.
\end{proposition}

\begin{proof}
\begin{itemize}
\item Suppose that $d \neq 0$ and $\dim(V) \geq n$. Let $h \in V$ such that $S(h) = 1$, $(h_{(2)}, \cdots, h_{(n)})$ be an independent system in $Ker(S)$ (this is possible since $\codim(Ker(S)) = 1$), and $(d, d_{(2)},\cdots,d_{(n)})$ be an independent system in $\mathbb{F}^n$. Consider the $V$-valued column matrix $\textbf{H} = \begin{bmatrix} h \\ h_{(2)} \\ \vdots \\ h_{(n)} \end{bmatrix}$ and the $\mathbb{F}$-valued square matrix $\textbf{D} = \begin{bmatrix} [c|ccc] d^1 & d_{(2)}^1 & \cdots & d_{(n)}^1  \\ \vdots & \vdots & \vdots & \vdots \\  d^n & d_{(2)}^n & \cdots & d_{(n)}^n  \end{bmatrix}$. Moreover, we set $\mathbf{A} = \begin{bmatrix} a_1 \\ \vdots \\ a_n \end{bmatrix} = \mathbf{D} \mathbf{H}$. \\
We have $T(a_1,\cdots,a_n) = (S(a_1), \cdots, S(a_n)) = (S(d^i h + \sum_{k=2}^n d_{(k)}^i h_{(k)}))_{i \in [\![1,n]\!]} = (d^i)_{i \in [\![1,n]\!]} = d$ since $h_{(2)}, \cdots, h_{(n)}$ belong to $Ker(S)$. \\
Let's show that $(a_1,\cdots,a_n)$ is free. Let $\lambda = (\lambda^1,\cdots,\lambda^n) \in \mathbb{F}^n$ such that $\sum_{i=1}^n  \lambda^i a_i = 0$. Consider $\mathbf{\Lambda} = \begin{bmatrix} \lambda^1 \\ \vdots \\ \lambda^n \end{bmatrix}$. Therefore we have $0 = \mathbf{\Lambda}^\top \mathbf{A} = \mathbf{\Lambda}^\top \mathbf{D} \mathbf{H} \in V$. At this point, we can complete $(h,h_{(2)}, \cdots, h_{(n)})$ into a Hamel basis of $V$, and denote by $(h^*,h_{(2)}^*, \cdots, h_{(n)}^*)$ the first $n$ linear forms of its dual basis. Applying these linear forms to $\mathbf{\Lambda}^\top \mathbf{D} \mathbf{H}$, we see that $\mathbf{\Lambda}^\top \mathbf{D} = \begin{bmatrix} [c|c|c] 0 & \cdots & 0 \end{bmatrix}$, therefore $\lambda = 0$ as $\mathbf{D}^\top$ is obviously invertible since $(h, d_{(2)},\cdots,d_{(n)})$ is an independent system in $\mathbb{F}^n$.
\item Suppose that $d = 0$ and $\dim(V) \geq n+1$. Let $h \in V$ such that $S(h) = 1$ and $(a_1, \cdots,a_n)$ be an independent system in $Ker(S)$. \\
We have $T(a_1,\cdots,a_n) = (S(a_1),\cdots,S(a_n)) = 0 \in \mathbb{F}^n$. \\
Moreover, we have by construction that $(a_1,\cdots,a_n)$ is free.
\end{itemize}
\end{proof}

\begin{remark}
In the previous proposition, if $d \neq 0$, the condition $\dim(V) \geq n$ is necessary for the existence of an independent n-tuple because the existence of an independent n-tuple implies that $\dim(V) \geq n$, and if $d=0$, the condition $\dim(V) \geq n+1$ is also necessary for the existence of an independent n-tuple in $T^{-1}(\{0\}$ because the existence of an independent n-tuple $(a_1,\cdots,a_n)$ in $T^{-1}(\{0\}$ implies that $S(a_i) = 0$ for all $i \in [\![1,n]\!]$, and since there exists $h \in V$ such that $S(h)=1$ because $S$ is non-zero, then $(a_1,\cdots,a_n,h)$ is free which implies $\dim(V) \geq n+1$.
\end{remark}

\begin{corollary}
\phantomsection
\label{CorT^(-1)(d)ContainsFrameEachCordTDependsSeparatelyOnOneVarCondDimL2X}
Let $h \in L^2(X,\mu;\mathbb{F}) \neq 0$. Define the linear operator 
\[ T : \begin{cases} L^2(X,\mu;\mathbb{F}^n) &\to \mathbb{F}^n \\ F = (f_x)_{x \in X} &\mapsto \int_X h(x)f_x d\mu(x) \end{cases}. \]
For all $d \in \mathbb{F}^n$, suppose that
\begin{itemize}
\item if $d \neq 0$, then $\dim(L^2(X,\mu;\mathbb{F})) \geq n$,
\item if $d = 0$ then $\dim(L^2(X,\mu;\mathbb{F})) \geq n+1$.
\end{itemize}
Then $T^{-1}(\{d\})$ contains a continuous frame $\Phi = (\varphi_x)_{x \in X}$.
\end{corollary}

\begin{proposition}
\phantomsection
\label{PropT^(-1)(d)ContainsFrame}
Let $h \in L^2(X,\mu;\mathbb{F}) \neq 0$. Define the linear operator 
\[ T : \begin{cases} L^2(X,\mu;\mathbb{F}^n) &\to \mathbb{F}^n \\ F = (f_x)_{x \in X} &\mapsto \int_X h(x)f_x d\mu(x) \end{cases}. \]
If there exists a measurable subset $Y \subset X$ such that $\dim(L^2(Y,\mu;\mathbb{F})) \geq n$ and $\mu((X \setminus Y)\cap h^{-1}(\mathbb{F}^*)) >0$, then for all $d \in \mathbb{F}^n$, $T^{-1}(\{d\})$ contains a continuous frame $\Phi = (\varphi_x)_{x \in X}$.
\end{proposition}

\begin{proof}
Since $\dim(L^2(Y,\mu;\mathbb{F})) \geq n$, there exists a continuous frame $(\varphi_y)_{y \in Y} \in \mathcal{F}_{(Y,\mu),n}^\mathbb{F}$. We extend $(\varphi_y)_{y \in Y}$ by setting 
\[ \varphi_x := \frac{\overline{h(x)}}{\lVert h \rVert_{L^2(X \setminus Y,\mu;\mathbb{F})}^2}\left(d - \int_Y h(y)\varphi_y d\mu(y) \right) \qquad \text{for all } x \in X \setminus Y \]
Let $\Phi = (\varphi_x)_{x \in X}$. We have 
\begin{align*}
 T(\Phi)  &= \int_X h(x)\varphi_x d\mu(x) \\
&= \int_Y h(y)\varphi_y d\mu(y) + \left( \int_{X \setminus Y} h(x) \frac{\overline{h(x)}}{\lVert h \rVert_{L^2(X\setminus Y,\mu;\mathbb{F})}^2} d\mu(x) \right) (d -  \int_Y h(y)\varphi_y d\mu(y)) \\
&= d.
\end{align*}
Moreover, $\Phi \in \mathcal{F}_{(X,\Sigma,\mu),n}^\mathbb{F}$ since we have only completed $(\varphi_y)_{y \in Y}$ by a function in $L^2(X \setminus Y,\mu;\mathbb{F}^n)$.
\end{proof}

\begin{proposition}
\label{PropW^(-1)(D)ContainsFrame}
Let $(X,\Sigma,\mu)$ be a measure space, $l \in \mathbb{N}^*$, $(X_j)_{j \in [\![1,l]\!]}$ a partition of $X$ by measurable subsets, and $h \in L^2(X,\mu;\mathbb{F})$ such that there exist a family $(Y_j)_{j \in [\![1,l]\!]}$ with $Y_j$ a measurable subset of $X_j$ for all $j \in [\![1,l]\!]$, $\mu((X_j \setminus Y_j) \cap h^{-1}(\mathbb{F}^*)) > 0$ for all $j \in [\![1,l]\!]$, and $\sum\limits_{j=1}^l \dim(L^2(Y_j,\mu;\mathbb{F})) \geq n$. \\
Define the operator 
\[ W : \begin{cases} L^2(X,\mu;\mathbb{F}^n) &\to \prod_{j \in [\![1,l]\!]} \mathbb{F}^n \\ F = (f_x)_{x \in X} &\mapsto (\int_{X_j} h(x)f_x d\mu(x))_{j \in [\![1,l]\!]} \end{cases}. \]
Then for all $D = (d_j)_{j \in [\![1,l]\!]} \in \prod_{j \in [\![1,l]\!]} \mathbb{F}^n$, $W^{-1}(\{D\})$ contains at least one continuous frame $\Phi \in \mathcal{F}_{(X,\Sigma,\mu),n}^\mathbb{F}$.
\end{proposition}

\begin{remark}
Proposition \ref{PropT^(-1)(d)ContainsFrame} results from proposition \ref{PropW^(-1)(D)ContainsFrame} by taking $l = 1$.
\end{remark}

\begin{remark}
Proposition \ref{PropW^(-1)(D)ContainsFrame} can be generalized to $l = +\infty$ or to partitions indexed by a general index set $J$ if we restrict to $D=0$ (due to convergence issues).
\end{remark}

\begin{proof}
For each $i \in [\![1,n]\!]$, let $e_i$ be the i-th vector of the standard basis of $\mathbb{F}^n$. Since $\sum\limits_{j=1}^l \dim(L^2(Y_j,\mu;\mathbb{F})) \geq n$, we can find distinct $j_1, \cdots, j_r \in [\![1,l]\!]$ such that for each $u \in [\![1,r]\!]$, $\dim(L^2(Y_{j_u},\mu;\mathbb{F})) \geq 1$ and $\sum_{u=1}^r \dim(L^2(Y_{j_u},\mu;\mathbb{F})) \geq n$. Take a partition $P_1, \cdots, P_r$ of $\{e_1, \cdots, e_n\}$ with $|P_u| \leq \dim(L^2(Y_{j_u},\mu;\mathbb{F}))$ for all $u \in [\![1,r]\!]$. \\
For all $u \in [\![1,r]\!]$, let $(g_u^p)_{p \in P_u}$ be an orthonormal family in $L^2(Y_{j_u},\mu;\mathbb{F})$ and define $(\varphi_x)_{x \in X_{j_u}}$ by 
\[ \varphi_y = \sum_{p \in P_u} g_u^p(y)p \qquad \text{for all y } \in Y_{j_u} \]
and 
\[ \varphi_x := \frac{\overline{h(x)}}{\lVert h \rVert_{L^2(X_{j_u} \setminus Y_{j_u},\mu;\mathbb{F})}^2}\left(d_{j_u} - \int_Y h(y)\varphi_y d\mu(y) \right) \qquad \text{for all } x \in X_{j_u} \setminus Y_{j_u}. \]
For all $j \notin \{j_u : u \in [\![1,r]\!]\}$, define $(\varphi_x)_{x \in X_j}$ by 
\[ \varphi_y = 0 \qquad \text{for all y } \in Y_j \]
and 
\[ \varphi_x := \frac{\overline{h(x)}}{\lVert h \rVert_{L^2(X_j \setminus Y_j,\mu;\mathbb{F})}^2}d_j \qquad \text{for all } x \in X_j \setminus Y_j. \]
Let $\Phi = (\varphi_x)_{x \in X}$. We have 
\begin{align*}
W(\Phi) &= \left(\int_{X_j} h(x)\varphi_x d\mu(x) \right)_{j \in [\![1,l]\!]} \\
&= \Bigg( \int_{Y_j} h(y)\varphi_y d\mu(y) \\
&\quad + \left( \int_{X_j \setminus Y_j} h(x) \frac{\overline{h(x)}}{\lVert h \rVert_{L^2(X_j \setminus Y_j,\mu;\mathbb{F})}^2} d\mu(x) \right) \left( d_j -  \int_{Y_j} h(y)\varphi_y d\mu(y) \right) \Bigg)_{j \in [\![1,l]\!]} \\
&= (d_j)_{j \in [\![1,l]\!]} = D.
\end{align*}
Moreover, $\Phi \in \mathcal{F}_{(X,\Sigma,\mu),n}^\mathbb{F}$ since 
\[ \forall v \in \mathbb{F}^n : \lVert v \rVert^2 = \sum_{u=1}^r \int_{Y_{j_u}} | \langle v  , \varphi_x \rangle |^2 d\mu(x) \leq \int_X | \langle v , \varphi_x \rangle |^2 d\mu(x) \]
and 
\begin{align*}
\int_X | \langle v , \varphi_x \rangle |^2 d\mu(x) &\leq \lVert v \rVert^2 + \left( \sum_{u=1}^r \frac{\lVert d_{j_u} - \int_Y h(y)\varphi_y d\mu(y) \rVert^2}{\lVert h \rVert_{L^2(X_{j_u} \setminus Y_{j_u},\mu;\mathbb{F})}^2} \right) \lVert v \rVert^2 \\
&\quad + \left( \sum_{j \notin \{j_u : u\in [\![1,r]\!]\}} \frac{\lVert d_j \rVert^2}{\lVert h \rVert_{L^2(X_j \setminus Y_j,\mu;\mathbb{F})}^2} \right) \lVert v \rVert^2 
\end{align*}
\end{proof}

\subsection{A quadratic equation}

\begin{proposition}
\label{Propq^(-1)(d)ContainsFrame}
Let $(X,\Sigma,\mu)$ be a $\sigma$-finite measure space, $b \in \mathbb{C}^n \setminus \{0\}$, $\epsilon > 0$ and $h \in L^\infty(X,\mu;\mathbb{C})$. \\
Consider 
\[ q : \begin{cases} L^2(X,\mu;\mathbb{C}^n) &\to \mathbb{C}^n \\ F = (f_x)_{x \in X} &\mapsto \int_X h(x) \langle b , f_x \rangle f_x d\mu(x) \end{cases}. \]
Then for all $d \in \mathbb{C}^n$ such that 
\begin{itemize} 
\item if $d \neq 0$, then there exists a measurable subset $Y \subseteq X$, two measurable subsets $B_1 \subseteq \{z \in \mathbb{C} : \operatorname{Re}(\langle b , d \rangle z) > \epsilon \text{ and } \operatorname{Im}(\langle b , d \rangle z) < - \epsilon \}$ and $B_2 \subseteq \{z \in \mathbb{C} : \operatorname{Re}(\langle b , d \rangle z) > \epsilon \text{ and } \operatorname{Im}(\langle b , d \rangle z) > \epsilon\}$ such that $\dim(L^2(Y,\mu;\mathbb{C})) \geq n$ and $\mu((X \setminus Y) \cap h^{-1}(B_1)),\mu((X \setminus Y) \cap h^{-1}(B_2))>0$,
\item if $d = 0$, then there exist a measurable subset $Y \subseteq X$ such that $\dim(L^2(Y,\mu;\mathbb{C})) \geq n$ and $h(x) < 0$ $\mu$-almost everywhere on $Y$, and [(two measurable subsets $B_1 \subseteq \{z \in \mathbb{C} : \operatorname{Re}(z) > 0 \text{ and } \operatorname{Im}(z) < 0 \}$ and $B_2 \subseteq \{z \in \mathbb{C} : \operatorname{Re}(z) > 0 \text{ and } \operatorname{Im}(z) > 0\}$ such that $\mu((X \setminus Y) \cap h^{-1}(B_1)),\mu((X \setminus Y) \cap h^{-1}(B_2))>0$) or (a measurable subset $B_3 \subseteq \{z \in \mathbb{C} : \operatorname{Re}(z) > 0 \text{ and } \operatorname{Im}(z) = 0\}$ such that $\mu((X \setminus Y) \cap h^{-1}(B_3))>0$)],
\end{itemize}
there exists a continuous frame $\Phi = (\varphi_x)_{x \in X} \in q^{-1}(\{d\})$.
\end{proposition}

\begin{proof}
\begin{itemize}
\item Suppose $d \neq 0$. Let $\widetilde{B_1} = (X \setminus Y) \cap h^{-1}(B_1)$ and $\widetilde{B_2} = (X \setminus Y) \cap h^{-1}(B_2)$. There is no loss in generality in assuming that $\mu(\widetilde{B_1})$ and $\mu(\widetilde{B_2})$ are finite since $\mu$ is $\sigma$-finite. Let $0 < a < \frac{\epsilon}{\lVert h \rVert_\infty^2 \lVert b \rVert^2}$. Since $\dim(L^2(Y,\mu;\mathbb{C})) \geq n$, we can pick an $a$-tight frame $(\varphi_y)_{y \in Y} \in \mathcal{F}_{(Y,\mu),n)}^\mathbb{C}$. Let $\widetilde{h}(x) = \left(\langle b , d \rangle - \int_Y h(y) |\langle b , \varphi_y \rangle|^2 d\mu(y)\right)h(x)$ for all $x \in X$. Notice that we have 
\[ \operatorname{Re}(\widetilde{h}(x)) > 0 \qquad \mu-\text{almost everywhere on } \widetilde{B_1}, \] 
\[ \operatorname{Im}(\widetilde{h}(x)) < 0 \qquad \mu-\text{almost everywhere on } \widetilde{B_1}, \]
\[ \operatorname{Re}(\widetilde{h}(x)) > 0 \qquad \mu-\text{almost everywhere on } \widetilde{B_2}, \]  
and 
\[ \operatorname{Im}(\widetilde{h}(x)) > 0 \qquad \mu-\text{almost everywhere on } \widetilde{B_2}. \] 
Let 
\[ A = \frac{1}{\frac{\langle - \operatorname{Im}(\tilde{h}), \operatorname{Re}(\tilde{h}) \rangle_{L^2(\tilde{B_1},\mu;\mathbb{C})}}{\lVert \operatorname{Im}(\tilde{h}) \rVert_{L^2(\tilde{B1},\mu;\mathbb{C})}^2} + \frac{\langle \operatorname{Im}(\tilde{h}), \operatorname{Re}(\tilde{h}) \rangle_{L^2(\tilde{B_2},\mu;\mathbb{C})}}{\lVert \operatorname{Im}(\tilde{h}) \rVert_{L^2(\tilde{B2},\mu;\mathbb{C})}^2}} > 0 \]
and 
\[ g(x) = \sqrt{A} \frac{\sqrt{-\operatorname{Im}(\widetilde{h}(x))}}{\lVert \operatorname{Im}(\widetilde{h}) \rVert_{L^2(\widetilde{B_1},\mu;\mathbb{C})}} 1_{\widetilde{B_1}}(x) + \sqrt{A} \frac{\sqrt{\operatorname{Im}(\widetilde{h}(x))}}{\lVert \operatorname{Im}(\widetilde{h}) \rVert_{L^2(\widetilde{B_2},\mu;\mathbb{C})}} 1_{\widetilde{B_2}}(x) \]
for all  $x \in X \setminus Y$. \\
Then it is easily seen that 
\begin{equation}
\label{eq1}
\left( \int_{X \setminus Y} h(x) |g(x)|^2 d\mu(x) \right) \left( -\langle b , d \rangle + \int_Y h(y) | \langle b , \varphi_y \rangle | ^2 d\mu(y) \right) = -1. 
\end{equation}
Consider $(\varphi_x)_{x \in X \setminus Y}$ defined by $\varphi_x = g(x)\left(-d + \int_Y h(y) \langle b , \varphi_y \rangle \varphi_y d\mu(y) \right)$ for all $x \in X \setminus Y$. Then $\Phi = (\varphi_x)_{x \in X} \in \mathcal{F}_{(X,\Sigma,\mu),n}^\mathbb{C}$ since we have only completed $(\varphi_y)_{y \in Y}$ by a function in $L^2(X \setminus Y,\mu;\mathbb{C}^n)$ . Moreover
\begin{align*}
q(\Phi) &= \int_X h(x) \langle b , \varphi_x \rangle \varphi_x d\mu(x) \\
&= \int_Y h(y) \langle b , \varphi_y \rangle \varphi_y d\mu(y) \\ 
&\quad + \int_{X \setminus Y} h(x) \langle b , g(x)\left(-d + \int_Y h(y) \langle b , \varphi_y \rangle \varphi_y d\mu(y) \right) \rangle \\
&\quad . g(x)\left(-d + \int_Y h(y) \langle b , \varphi_y \rangle \varphi_y d\mu(y) \right) d\mu(x) \\
&= \int_Y h(y) \langle b , \varphi_y \rangle \varphi_y d\mu(y) \\
&\quad - \underbrace{\left[ \int_{X \setminus Y} h(x) |g(x)|^2 \left(-\langle b , d \rangle + \int_Y h(y) |\langle b , \varphi_y \rangle |^2 d\mu(y) \right) d\mu(x) \right]}_{=-1} d \\
&\quad + \underbrace{\left[ \int_{X \setminus Y} h(x) |g(x)|^2 \left(-\langle b , d \rangle + \int_Y h(y) |\langle b , \varphi_y \rangle |^2 d\mu(y) \right) d\mu(x) \right]}_{=-1} \\
&\quad . \left(\int_Y h(y) \langle b , \varphi_y \rangle \varphi_y d\mu(y) \right) \\
&= d
\end{align*}
using equality \ref{eq1}.

\item Suppose $d = 0$. Let $\widetilde{B_1} = (X \setminus Y) \cap h^{-1}(B_1)$ and $\widetilde{B_2} = (X \setminus Y) \cap h^{-1}(B_2)$. \\
Suppose first that $\mu(\widetilde{B_1}),\mu(\widetilde{B_2})>0$. There is no loss in generality in assuming that $\mu(\widetilde{B_1})$ and $\mu(\widetilde{B_2})$ are finite since $\mu$ is $\sigma$-finite.  Since $\dim(L^2(Y,\mu;\mathbb{C})) \geq n$, we can pick a frame $(\varphi_y)_{y \in Y} \in \mathcal{F}_{(Y,\mu),n}^\mathbb{C}$. Let $\widetilde{h}(x) = - \left( \int_Y h(y) |\langle b , \varphi_y \rangle|^2 d\mu(y) \right) h(x)$ for all $x \in X$. Notice that we have 
\[ \operatorname{Re}(\widetilde{h}(x)) > 0 \qquad \mu-\text{almost everywhere on } \widetilde{B_1}, \] 
\[ \operatorname{Im}(\widetilde{h}(x)) < 0 \qquad \mu-\text{almost everywhere on } \widetilde{B_1}, \]
\[ \operatorname{Re}(\widetilde{h}(x)) > 0 \qquad \mu-\text{almost everywhere on } \widetilde{B_2}, \]  
and 
\[ \operatorname{Im}(\widetilde{h}(x)) > 0 \qquad \mu-\text{almost everywhere on } \widetilde{B_2}. \] 
Let 
\[ A = \frac{1}{\frac{\langle - \operatorname{Im}(\tilde{h}), \operatorname{Re}(\tilde{h}) \rangle_{L^2(\tilde{B_1},\mu;\mathbb{C})}}{\lVert \operatorname{Im}(\tilde{h}) \rVert_{L^2(\tilde{B1},\mu;\mathbb{C})}^2} + \frac{\langle \operatorname{Im}(\tilde{h}), \operatorname{Re}(\tilde{h}) \rangle_{L^2(\tilde{B_2},\mu;\mathbb{C})}}{\lVert \operatorname{Im}(\tilde{h}) \rVert_{L^2(\tilde{B2},\mu;\mathbb{C})}^2}} > 0 \]
and 
\[ g(x) = \sqrt{A} \frac{\sqrt{-\operatorname{Im}(\widetilde{h}(x))}}{\lVert \operatorname{Im}(\widetilde{h}) \rVert_{L^2(\widetilde{B_1},\mu;\mathbb{C})}} 1_{\widetilde{B_1}}(x) + \sqrt{A} \frac{\sqrt{\operatorname{Im}(\widetilde{h}(x))}}{\lVert \operatorname{Im}(\widetilde{h}) \rVert_{L^2(\widetilde{B_2},\mu;\mathbb{C})}} 1_{\widetilde{B_2}}(x) \]
for all $x \in X \setminus Y$. \\
Then it is easily seen that 
\begin{equation}
\label{eq2}
\left( \int_{X \setminus Y} h(x) |g(x)|^2  d\mu(x) \right) \left(\int_Y h(y) | \langle b , \varphi_y \rangle | ^2 d\mu(y) \right) = -1. 
\end{equation}
Consider $(\varphi_x)_{x \in X \setminus Y}$ defined by $\varphi_x = g(x)\int_Y h(y) \langle b , \varphi_y \rangle \varphi_y d\mu(y)$ for all $x \in X \setminus Y$. Then $\Phi = (\varphi_x)_{x \in X} \in \mathcal{F}_{(X,\Sigma,\mu),n}^\mathbb{C}$ since we have only completed $(\varphi_y)_{y \in Y}$ by a function in $L^2(X \setminus Y,\mu;\mathbb{C}^n)$ . Moreover
\begin{align*}
q(\Phi) &= \int_X h(x) \langle b , \varphi_x \rangle \varphi_x d\mu(x) \\
&= \int_Y h(y) \langle b , \varphi_y \rangle \varphi_y d\mu(y) \\ 
&\quad + \int_{X \setminus Y} h(x) \langle b , g(x)\left(\int_Y h(y) \langle b , \varphi_y \rangle \varphi_y d\mu(y) \right) \rangle \\
&\quad . g(x)\left(\int_Y h(y) \langle b , \varphi_y \rangle \varphi_y d\mu(y) \right) d\mu(x) \\
&= \int_Y h(y) \langle b , \varphi_y \rangle \varphi_y d\mu(y) \\
&\quad + \underbrace{\left[ \int_{X \setminus Y} h(x) |g(x)|^2 \left(\int_Y h(y) |\langle b , \varphi_y \rangle |^2 d\mu(y) \right) d\mu(x) \right]}_{=-1} \\
&\quad . \left(\int_Y h(y) \langle b , \varphi_y \rangle \varphi_y d\mu(y) \right) \\
&= 0
\end{align*}
using equality \ref{eq2}. \\
Now let  $\widetilde{B_3} = (X \setminus Y) \cap h^{-1}(B_3)$ and suppose instead that $\mu(\widetilde{B_3})>0$. There is no loss in generality in assuming that $\mu(\widetilde{B_3})$ is finite since $\mu$ is $\sigma$-finite.  Since $\dim(L^2(Y,\mu;\mathbb{C})) \geq n$, we can pick a frame $(\varphi_y)_{y \in Y} \in \mathcal{F}_{(Y,\mu),n}^\mathbb{C}$. Let $\widetilde{h}(x) = - \left( \int_Y h(y) |\langle b , \varphi_y \rangle|^2 d\mu(y) \right) h(x)$ for all $x \in X$. Notice that we have 
\[ \widetilde{h}(x) > 0 \qquad \mu-\text{almost everywhere on } \widetilde{B_3}. \] 
Let 
\[ g(x) = \frac{\sqrt{\widetilde{h}(x)}}{\lVert \widetilde{h} \rVert_{L^2(\widetilde{B_3},\mu;\mathbb{C})}} 1_{\widetilde{B_3}}(x) \text{ for all } x \in X \setminus Y \]
Then it is easily seen that 
\begin{equation}
\label{eq3}
\left( \int_{X \setminus Y} h(x) |g(x)|^2  d\mu(x) \right) \left(\int_Y h(y) | \langle b , \varphi_y \rangle | ^2 d\mu(y) \right) = -1. 
\end{equation}
Consider $(\varphi_x)_{x \in X \setminus Y}$ defined by $\varphi_x = g(x)\int_Y h(y) \langle b , \varphi_y \rangle \varphi_y d\mu(y)$ for all $x \in X \setminus Y$. Then $\Phi = (\varphi_x)_{x \in X} \in \mathcal{F}_{(X,\Sigma,\mu),n}^\mathbb{C}$ since we have only completed $(\varphi_y)_{y \in Y}$ by a function in $L^2(X \setminus Y,\mu;\mathbb{C}^n)$ . Moreover we can prove that $q(\Phi)=0$ as before using equality \ref{eq3}.
\end{itemize}
\end{proof}

\section{Examples of subspsaces in which the space of continuous frames is relatively dense}
\label{SectionExamples}

\begin{corollary}
Let $n \in \mathbb{N}^*$ and $T : L^2(X,\mu;\mathbb{F}^n) \to \mathbb{F}$ be a non-zero linear form. \\
Suppose that $\dim(L^2(X,\mu;\mathbb{F})) \geq n$. \\
Then since $T^{-1}(\{d\})$ is affine and by corollaries \ref{CorSt(n,H)CapSDenseInS} and \ref{CorT^(-1)(d)ContainsFrameTArbitraryLinearFormCondDimL2X}, for all $d \neq 0$, $\mathcal{F}_{(X,\Sigma,\mu),n}^\mathbb{F} \cap T^{-1}(\{d\})$ is dense in $T^{-1}(\{d\})$
\end{corollary}

\begin{corollary}
Let $(X,\Sigma,\mu)$ be a measure space, $d \in \mathbb{C}^n$, and $h \in L^2(X,\mu;\mathbb{F})$ such that $T^{-1}(\{d\})$ contains a continuous frame $\Phi = (\varphi_x)_{x \in X}$ (see for instance corollary \ref{CorT^(-1)(d)ContainsFrameEachCordTDependsSeparatelyOnOneVarCondDimL2X} and proposition \ref{PropT^(-1)(d)ContainsFrame}), where 
\[ T : \begin{cases} L^2(X,\mu;\mathbb{F}^n) &\to \mathbb{F}^n \\ F = (f_x)_{x \in X} &\mapsto \int_X h(x)f_x d\mu(x) \end{cases}. \]
Then since $T^{-1}(\{d\})$ is affine, by corollary \ref{CorSt(n,H)CapSDenseInS}, $\mathcal{F}_{(X,\Sigma,\mu),n}^\mathbb{F} \cap T^{-1}(\{d\})$ is dense in $T^{-1}(\{d\})$.
\end{corollary}

\begin{corollary}
Let $(X,\Sigma,\mu)$ be a measure space, $l \in \mathbb{N}^*$, $(X_j)_{j \in [\![1,l]\!]}$ a partition of $X$ by measurable subsets, $h \in L^2(X,\mu;\mathbb{F})$, and $D \in \mathbb{F}^n$ such that $W^{-1}(\{D\})$ contains a continuous frame $\Phi = (\varphi_x)_{x \in X}$ (see for instance proposition \ref{PropW^(-1)(D)ContainsFrame}), where 
\[ W : \begin{cases} L^2(X,\mu;\mathbb{F}^n) &\to \prod_{j \in [\![1,l]\!]} \mathbb{F}^n \\ F = (f_x)_{x \in X} &\mapsto (\int_{X_j} h(x)f_x)_{j \in [\![1,l]\!]} \end{cases}. \]
Then since since $W^{-1}(\{d\})$ is affine, and by corollary \ref{CorSt(n,H)CapSDenseInS}, $\mathcal{F}_{(X,\Sigma,\mu),n}^\mathbb{F} \cap W^{-1}(\{D\})$ is dense in $W^{-1}(\{D\})$.
\end{corollary}

\begin{remark}
Consider the function $q$ of proposition \ref{Propq^(-1)(d)ContainsFrame}. If $\Phi = (\varphi_x)_{x \in X} \in q^{-1}(\{0\})$, then we also have $\Phi \in q^{-1}(\{0\}) \cap (q^{-1}(\{0\}) - \Phi)$ since $q(2\Phi)=0$.
\end{remark}

\begin{proposition}
\label{Propq-1(0)Cap(q-1(0)-Phi)PolynomialPathConnected}
Consider the function $q$ of proposition \ref{Propq^(-1)(d)ContainsFrame}. Let $\Phi = (\varphi_x)_{x \in X} \in q^{-1}(\{0\})$ and $U = (u_x)_{x \in X} \in q^{-1}(\{0\}) \cap (q^{-1}(\{0\}) - \Phi)$. Then for all $\lambda,\mu \in \mathbb{R}$, $\lambda \Phi + \mu U \in q^{-1}(\{0\}) \cap (q^{-1}(\{0\}) - \Phi)$. \\
In particular,  $q^{-1}(\{0\}) \cap (q^{-1}(\{0\}) - \Phi)$ is a star domain relatively to $\Phi$.
\end{proposition}

\begin{proof}
Let $\lambda,\mu \in \mathbb{R}$. Let $s$ be the sesquilinear form $\begin{cases} L^2(X,\mu;\mathbb{C}^n) &\to \mathbb{C}^n \\ (F,G) &\mapsto \int_X \langle b , g_x \rangle f_x d\mu(x) \end{cases}$. We have
\begin{align*} 
q(\lambda \Phi + \mu U) &= \lambda^2 \underbrace{q(\Phi)}_{0} + \lambda \mu ( s(\Phi,U) + s(U,\Phi)) + \mu^2 \underbrace{q(U)}_{0} \\
&= \lambda \mu (\underbrace{q(\Phi+U)}_{0} - \underbrace{q(\Phi)}_{0} - \underbrace{q(U)}_{0}) \\
&= 0.
\end{align*}
We can show similarly that $q((\lambda+1)\Phi + \mu U)=0$.
\end{proof}

\begin{corollary}
Let $(X,\Sigma,\mu)$ be a $\sigma$-finite measure space, $b \in \mathbb{C}^n \setminus \{0\}$, and $h \in L^\infty(X,\mu;\mathbb{C})$ such that $q^{-1}(\{0\})$ contains a continuous frame $\Phi = (\varphi_x)_{x \in X}$ (see for instance proposition \ref{Propq^(-1)(d)ContainsFrame}), where 
\[ q : \begin{cases} L^2(X,\mu;\mathbb{C}^n) &\to \mathbb{C}^n \\ F = (f_x)_{x \in X} &\mapsto \int_X h(x) \langle b , f_x \rangle f_x d\mu(x) \end{cases}. \]
Then by proposition \ref{Propq-1(0)Cap(q-1(0)-Phi)PolynomialPathConnected} and corollary \ref{CorSt(n,H)CapSDenseInS}, $\mathcal{F}_{(X,\Sigma,\mu),n}^\mathbb{C} \cap q^{-1}(\{0\}) \cap (q^{-1}(\{0\}) - \Phi)$ is dense in $q^{-1}(\{0\}) \cap (q^{-1}(\{0\}) - \Phi)$.
\end{corollary}

\section*{Acknowledgement}
The first author is financially supported by the \textit{Centre National pour la Recherche Scientifique et Technique} of Morocco.

\section*{Conflict of interest}
On behalf of all authors, the corresponding author states that there is no conflict of interest.

\nocite{*}
\bibliographystyle{plain}
\bibliography{references}

\begin{thebibliography}{10}

\bibitem{Agrawal}
D.~Agrawal.
\newblock The complete structure of linear and nonlinear deformations of frames
  on a hilbert space.
\newblock Master's thesis, 2016.

\bibitem{AgrawalKnisley}
D.~Agrawal and J.~Knisley.
\newblock Fiber bundles and parseval frames, 2015.
\newblock arxiv:1512.03989.

\bibitem{AliAntoineGazeau}
S.~T. Ali, J.~P. Antoine, and J.~P. Gazeau.
\newblock Continuous frames in hilbert space.
\newblock {\em Annals of physics}, 222(1):1--37, 1993.

\bibitem{BalanCasazzaHeilLandau1}
R.~Balan, P.~G. Casazza, C.~Heil, and Z.~Landau.
\newblock Density, overcompleteness, and localization of frames. i. theory.
\newblock {\em Journal of Fourier Analysis and Applications}, 12(2):105--143,
  2006.

\bibitem{BalanCasazzaHeilLandau2}
R.~Balan, P.~G. Casazza, C.~Heil, and Z.~Landau.
\newblock Density, overcompleteness, and localization of frames. ii. gabor
  systems.
\newblock {\em Journal of Fourier Analysis and Applications}, 12(3):307--344,
  2006.

\bibitem{BardelliMennucci}
E.~Bardelli and A.~C.~G. Mennucci.
\newblock Probability measures on infinite-dimensional stiefel manifolds.
\newblock {\em Journal of Geometric Mechanics}, 9(3):291--316, 2017.

\bibitem{Bownik}
M.~Bownik.
\newblock Connectivity and density in the set of framelets.
\newblock {\em Mathematical Research Letters}, 14(2):285--293, 2017.

\bibitem{CahillMixonStrawn}
J.~Cahill, D.~Mixon, and N.~Strawn.
\newblock Connectivity and irreducibility of algebraic varieties of finite unit
  norm tight frames.
\newblock {\em SIAM Journal on Applied Algebra and Geometry}, 1(1):38--72,
  2017.

\bibitem{Casazza}
P.~G. Casazza.
\newblock The art of frame theory.
\newblock {\em Taiwanese Journal of Mathematics}, 4(2):129--201, 2000.

\bibitem{CasazzaHanLarson}
P.~G. Casazza, D.~Han, and D.~R. Larson.
\newblock Frames for banach spaces.
\newblock {\em Contemporary Mathematics}, 247:149--182, 1999.

\bibitem{CasazzaKutyniok}
P.~G. Casazza and G.~Kutyniok.
\newblock {\em Finite frames, theory and applications}.
\newblock Applied and Numerical Harmonic Analysis. Birkh{\"a}user/Springer,
  2013.

\bibitem{Christensen}
O.~Christensen.
\newblock {\em An Introduction to Frames and Riesz Bases}.
\newblock Applied and Numerical Harmonic Analysis. Birkhäuser/Springer, second
  edition, 2016.

\bibitem{ChristensenDengHeil}
O.~Christensen, B.~Deng, and C.~Heil.
\newblock Density of gabor frames.
\newblock {\em Applied and Computational Harmonic Analysis}, 7(3):292--304,
  1999.

\bibitem{DaubechiesGrossmannMeyer}
I.~Daubechies, A.~Grossmann, and Y.~Meyer.
\newblock Painless nonorthogonal expansions.
\newblock {\em J. Math. Phys.}, 27(5):1271--1283, 1986.

\bibitem{DuffinSchaeffer}
R.~J. Duffin and A.~C. Schaeffer.
\newblock A class of nonharmonic fourier series.
\newblock {\em Transactions of the American Mathematical Society},
  72(2):341--366, 1952.

\bibitem{DykemaStrawn}
K.~Dykema and N.~Strawn.
\newblock Manifold structure of spaces of spherical tight frames.
\newblock {\em Int. J. Pure Appl. Math.}, 28(2):217--256, 2006.

\bibitem{FrankLarson}
M.~Frank and D.~R. Larson.
\newblock Frames in hilbert c*-modules and c*-algebras.
\newblock {\em Journal of Operator Theory}, 48(2):273--314, 2002.

\bibitem{Gabor}
D.~Gabor.
\newblock Theory of communications.
\newblock {\em Jour. Inst. Elec. Eng. (London)}, 93(3):429--457, 1946.

\bibitem{GarrigosHernandezSikicSoriaWeissWilson}
G.~Garrigos, E.~Hernandez, H.~{\v S}iki{\'c}, F.~Soria, G.~Weiss, and
  E.~Wilson.
\newblock Connectivity in the set of tight frame wavelets (tfw).
\newblock {\em Glasnik matemati{\v c}ki}, 38(1):75--98, 2003.

\bibitem{HanLarson}
D.~Han and D.~R. Larson.
\newblock On the orthogonality of frames and the density and connectivity of
  wavelet frames.
\newblock {\em Acta Applicandae Mathematicae}, 107(1-3):211--222, 2009.

\bibitem{HarmsMennucci}
P.~Harms and A.~C.~G. Mennucci.
\newblock Geodesics in infinite dimensional stiefel and grassmann manifolds.
\newblock {\em Comptes Rendus Mathematique}, 350(15-16):773--776, 2012.

\bibitem{Hatcher}
A.~Hatcher.
\newblock {\em Algebraic topology}.
\newblock Cambridge University Press, 2005.

\bibitem{Heil}
C.~Heil.
\newblock History and evolution of the density theorem for gabor frames.
\newblock {\em Journal of Fourier Analysis and Applications}, 13(2):113--166,
  2007.

\bibitem{Henkel}
O.~Henkel.
\newblock Sphere-packing bounds in the grassmann and stiefel manifolds.
\newblock {\em IEEE Transactions on Information Theory}, 51(10):3445--3456,
  2005.

\bibitem{Husemoller}
D.~Husemoller.
\newblock {\em Fibre Bundles}.
\newblock Springer-Verlag, third edition, 1994.

\bibitem{James}
I.~M. James.
\newblock {\em The topology of Stiefel manifolds}.
\newblock Cambridge University Press, 1976.

\bibitem{JurdjevicMarkinaLeite}
V.~Jurdjevic, I.~Markina, and F.~S. Leite.
\newblock Extremal curves on stiefel and grassmann manifolds.
\newblock {\em J. Geom. Anal.}, 30(4):3948--3978, 2020.

\bibitem{LabateWilson}
D.~Labate and E.~Wilson.
\newblock Connectivity in the set of gabor frames.
\newblock {\em Applied and Computational Harmonic Analysis}, 18(1):123--136,
  2005.

\bibitem{NeedhamShonkwiler}
T.~Needham and C.~Shonkwiler.
\newblock Symplectic geometry and connectivity of spaces of frames.
\newblock {\em Advances in Computational Mathematics}, 47(1 : 5), 2021.

\bibitem{Speegle}
D.~M. Speegle.
\newblock The s-elementary wavelets are path-connected.
\newblock {\em Proc. Amer. Math. Soc.}, 127:223--233, 1999.

\bibitem{Strawn2}
N.~Strawn.
\newblock Geometry and constructions of finite frames.
\newblock Master's thesis, 2007.

\bibitem{Strawn1}
N.~Strawn.
\newblock Finite frame varieties: nonsingular points, tangent spaces, and
  explicit local parameterizations.
\newblock {\em J. Fourier Anal. Appl.}, 17(5):821--853, 2011.

\bibitem{Strawn3}
N.~Strawn.
\newblock {\em Geometric structures and optimization on spaces of finite
  frames}.
\newblock PhD thesis, 2011.

\bibitem{Waldron.}
S.~F.~D. Waldron.
\newblock {\em An introduction to finite tight frames}.
\newblock Applied and Numerical Harmonic Analysis. Birkh{\"a}user/Springer,
  2018.

\end{thebibliography}

\Addresses

\end{document}